\documentclass[10pt,a4paper]{amsart}
\usepackage{fullpage,epsfig, amsmath,xspace}
\usepackage{amssymb,amscd}
\usepackage[all,cmtip]{xy} %\CompileMatrices
\usepackage{hyperref}
\newtheorem{thm}{Theorem}[section]
\newtheorem{lem}[thm]{Lemma}
\newtheorem{prop}[thm]{Proposition}
\newtheorem{cor}[thm]{Corollary}
\newtheorem{add}[thm]{Addendum}

\theoremstyle{definition}
\newtheorem{rem}[thm]{Remark}
\newtheorem{defn}[thm]{Definition}
\newtheorem{ex}[thm]{Example}

\newdir{ >}{{}*!/-7pt/@{>}}

\def\Z{\mathbb{Z}}

\def\bfn{\mathbf{n}}
\def\bfm{\mathbf{m}}
\def\calc{\mathcal{C}}
\def\R{\mathcal{R}}
\def\K{\mathcal{K}}
\def\M{\mathcal{M}}

\def\lra{\longrightarrow}
\newcommand{\we}{\overset{\sim}{\longrightarrow}}
\def\leq{\leqslant}
\def\geq{\geqslant}
\def\id{\mathrm{id}}
\def\ra{\rightarrow}
\DeclareMathOperator{\colim}{colim}
\DeclareMathOperator{\hocolim}{hocolim}
\DeclareMathOperator{\ob}{ob}

\def\ie{\emph{i.e.}}

\usepackage{xspace}
\usepackage{color}%REMOVE THIS ONCE EDITING IS OVER

\def\bc{\text{rig category}\xspace}
\newcommand{\bcs}{\text{rig categories}\xspace}

\newcommand{\sbc}{\text{strictly bimonoidal category}\xspace}

\begin{document}
\title[%Two-vector bundles define a form of elliptic cohomology
Stable bundles over rig categories]{
Stable bundles over rig categories}
\author{Nils A. Baas, Bj{\o}rn Ian Dundas, Birgit Richter and John Rognes}
\address{Department of Mathematical Sciences, NTNU,
  7491 Trondheim, Norway}
\email{baas@math.ntnu.no}
\address{Department of Mathematics, University of Bergen,
  5008 Bergen, Norway}
\email{dundas@math.uib.no}
\address{Fachbereich Mathematik der Universit\"at Hamburg,
  20146 Hamburg, Germany}
\email{richter@math.uni-hamburg.de}
\address{Department of Mathematics, University of Oslo,
  0316 Oslo, Norway}
\email{rognes@math.uio.no}
\date{\today}
\thanks{
The first author would like to thank the Institute for Advanced Study,
Princeton, for their hospitality and support during his stay in the
spring of 2007.
Part of the work was done while the second author was on sabbatical at
Stanford University, whose hospitality and stimulating environment is
gratefully acknowledged.
The third author thanks %%%140809
the SFB 676 for support and %%%140809
the topology group
in Sheffield for stimulating discussions on the subject.}
\hyphenation{ca-te-go-ries ca-te-go-ry}

\keywords{Algebraic $K$-theory, topological $K$-theory, 2-vector bundles,
  elliptic cohomology, bimonoidal categories, bipermutative categories}
\subjclass[2000]{Primary 19D23, 55R65; Secondary 19L41, 18D10}
\begin{abstract}
The point of this paper is to prove the conjecture that virtual
$2$-vector  bundles are classified by $K(ku)$, the algebraic
$K$-theory of  topological $K$-theory.  Hence, by the work of Ausoni
and the  fourth author, virtual $2$-vector bundles give us a geometric
cohomology theory of the same telescopic complexity as elliptic cohomology.

% The main technical step is showing that for well-behaved small \bcs~$\R$
% (also known as bimonoidal categories) the algebraic $K$-theory
% space, $K(H\R)$,  of the $K$-theory ring spectrum $H\R$ of $\R$ is
% equivalent  to
% $\K(\R) \simeq \Z \times |BGL(\R)|^+$, where $GL(\R)$ is the
% monoidal category of weakly invertible matrices over~$\R$.

The main technical step is showing that for well-behaved small
\bcs~$\R$ (also known as bimonoidal categories) the algebraic
$K$-theory  space, $K(H\R)$, of the ring spectrum $H\R$ associated
to $\R$ is equivalent  to  $\K(\R) \simeq \Z \times |BGL(\R)|^+$,
where $GL(\R)$ is the  monoidal category of weakly invertible
matrices over~$\R$.

The title refers to the sharper result that $BGL(\R)$ is equivalent to
$BGL(H\R)$.
If $\pi_0\R$ is a ring this is almost formal, and our approach is to
replace  $\R$ by  a ring completed version,
$\bar\R$, provided by \cite{BDRR1} with $H\R\simeq H\bar\R$ and
$\pi_0\bar\R$  the ring completion of $\pi_0\R$.  The remaining step
is then to  show that ``stable $\R$-bundles'' and ``stable
$\bar\R$-bundles''  are the same, which is done by a hands-on
contraction of a  custom-built model for the difference between
$BGL(\R)$ and  $BGL(\bar\R)$.

% In particular, this proves the conjecture from \cite{BDR} that $K(ku)$
% is the $K$-theory of the $2$-category of complex $2$-vector spaces.
% Hence, the work of Christian Ausoni and the fourth author on $K(ku)$
% \cite{AR, A} shows that the theory of virtual $2$-vector
% bundles as in \cite[Theorem 4.10]{BDR} qualifies as a form of elliptic
% cohomology theory.

\end{abstract}
\maketitle
\section{Introduction and main result}
In telescopic complexity $0$, $1$ and $\infty$ there are cohomology
theories that possess a geometric definition: de Rham cohomology of
manifolds is given in terms of differential forms, cohomology
classes in real and complex $K$-theory are classes of virtual vector
bundles, and complex cobordism has a geometric definition \emph{per se}.
In order to understand phenomena of intermediate telescopic complexity, it
is desirable to have geometric interpretations for such cohomology
theories as well.

In \cite{BDR} it was conjectured that virtual $2$-vector bundles
provide a  geometric interpretation of a cohomology theory of
telescopic  complexity $2$ which qualifies as a form of elliptic
cohomology.  More precisely, it was conjectured that the algebraic
$K$-theory of a commutative \bc $\R$ is equivalent to the
algebraic $K$-theory of the  ring spectrum associated with $\R$.
The case of virtual $2$-vector  bundles arises when $\R$ is the
category of finite dimensional  complex vector spaces, with $\oplus$
and $\otimes_\mathbf C$ as  sum and multiplication.  This, together
with the analysis of the  $K$-theory of complex topological
$K$-theory due to Ausoni and
the fourth  author \cite{AR, A}, and the %%%020909(now proven)
Quillen--Lichtenbaum
conjecture for the integers, gives the desired relation to elliptic
cohomology.

In this paper we prove the conjecture from \cite{BDR}.
%%%%%%%%% Inserted 060511 NAB/JR:
This work was motivated by the study of extended topological
quantum field theories and the search for a geometric definition
of elliptic cohomology, see \cite{BDR}, \cite{F}, \cite{K} and \cite{L}.
%%%%%%%%%
% Connections between $2$-vector bundles and elliptic objects/quantum
% field theories have been a source of speculations long before
% \cite{BDR} considered the issue, see for instance \cite{K} and \cite{F}.
%%%%%%%%%
%the very interesting
%papers by Kapranov \cite{K} and Freed \cite{F}.
%The authors were not aware of these papers when they wrote \cite{BDR} and
%are grateful to Jack Morava for pointing them out to us.

Let $\R$ be a %commutative
\bc (also known as a bimonoidal category), \ie,
a category with two operations $\oplus$ and $\otimes$
satisfying the axioms of a %commutative
rig (ring without negative elements) up to coherent natural
isomorphisms%, see  %\cite{L} or
%Definitions~\ref{def:igradedbim} and~\ref{def:igradedbc}
%below for details
.  In analogy with Quillen's
definition of the algebraic $K$-theory space $K(A)=\Omega
B(\coprod_n BGL_n(A))$ of a ring $A$, the
algebraic $K$-theory of $\R$ was defined in \cite{BDR} as
$\K(\R)=\Omega B(\coprod_n |BGL_n(\R)|)\simeq\Z\times |BGL(\R)|^+$
where $B$  and $GL_n$ are
versions of the bar construction and
the general linear group appropriate for \bcs.

On the other hand, forgetting the multiplicative structure, $\R$ has an
underlying symmetric monoidal category, and so it makes sense to speak
about its $K$-theory spectrum $H\R$ with respect to $\oplus$. The
$K$-theory spectrum construction $H\R$ is a direct extension of the usual
Eilenberg--Mac\,Lane construction, and can, since $\R$ is a \bc, be endowed
with the structure of a strict ring spectrum, for instance through the
model given by Elmendorf and Mandell in \cite{EM}.  Hence, we may speak
about its algebraic $K$-theory space $K(H\R)$.  We prove that, under
certain mild restrictions on $\R$, there is an equivalence $$\K(\R)\simeq
K(H\R) \,.$$

In the special situation where $\R$ is a ring
(\ie, $\R$ is discrete as a category and has negative
elements), this is the standard assertion that the $K$-theory of a ring
is equivalent to the $K$-theory of its associated Eilenberg-Mac\,Lane
spectrum.  The key difficulty in establishing the equivalence above
lies in proving that the lack of negative elements makes no difference
for algebraic $K$-theory, even for %commutative
\bcs.

More precisely, we prove the following result:

\begin{thm} \label{thm:main}
Let $(\R, \oplus, 0_\R, \otimes,  1_\R)$ be a small  topological %commutative
\bc satisfying the following conditions:
\begin{enumerate}
\item
All morphisms in $\R$ are isomorphisms, %%%140809
\ie, $\R$ is a groupoid.
\item
For every object $X \in \R$ the translation functor
$X \oplus (-)$ is faithful.
%, and if there exists a morphism from $X \oplus A$
%to $A$, then $X$ is equal to the zero object.
\end{enumerate}

Then $|BGL(\R)|$ and $B{GL}(H\R)$ are weakly equivalent.
Hence, the algebraic $K$-theory space of $\R$ as a \bc,
$$\K(\R) =
\Omega B\Big(\coprod_{n\geq0} |BGL_n(\R)|\Big)
\simeq \Z \times |BGL(\R)|^+ \,,$$
is weakly equivalent to the algebraic $K$-theory space of the strict
ring spectrum associated to $\R$,
$$K(H\R)=\Omega B\Big(\coprod_{n\geq0} B{GL}_n(H\R)\Big) \simeq
\Z \times B{GL}(H\R)^+ \,.$$
\end{thm}

\begin{add}
In particular, if $\R$ is the category of finite dimensional complex
vector spaces, the theorem states that stable $2$-vector bundles, in
the sense of \cite{BDR} are
classified by $BGL(ku)$, where $ku=H\R$ is the connective complex $K$-theory spectrum
with $\pi_*ku=\Z[u]$, $|u|=2$.
%We will prove that $|BGL(\R)|$ and $B{GL}(H\R)$ are weakly equivalent
%before applying the plus construction.
\end{add}

In contrast to $K(H\R)$, which is built in a
two-stage process, the $K$-theory of the (strictly) bimonoidal
category $\R$ is built using both monoidal structures at once, so in
this sense $\K(\R)$ is a model that is easier to understand and
handle than $K(H\R)$.

The conditions (1) and (2) on $\R$ are not restrictive for the applications we
have in mind, and are associated with the fact that in \cite{BDRR1} we
chose  to work with variants of the Grayson--Quillen
model for $K$-theory. Probably, the restrictions can be removed if one
uses another technological platform.

Among those \bcs that satisfy the requirements of
Theorem \ref{thm:main} are the following `standard' ones, usually
considered in the context of $K$-theory constructions.

\begin{itemize}
\item
If $\R$ is the discrete category (having only identity morphisms)
with objects the elements of a %commutative
ring with unit, $R$, then $H\R$ is the
Eilenberg-Mac\,Lane spectrum $HR$.
\item
The sphere spectrum $S$  is the algebraic $K$-theory spectrum of the small
\bc of finite sets~$\mathcal{E}$. The objects of
$\mathcal{E}$ are the finite sets $\bfn = \{1,\ldots,n\}$ for %jr1804
$n \geq 0$, with
%% JR3-1: made 0 bold
the convention that $\mathbf{0}$ is the empty set. Morphisms from
$\bfn$ to $\bfm$ only exist %jr1804
for $n=m$, and in
this case they constitute %jr1804
the symmetric group on $n$ letters. The
algebraic $K$-theory of $S$ is equivalent to Waldhausen's $A$-theory of
a point $A(*)$ \cite{1126}, and so gives information about
diffeomorphisms of high dimensional disks.
Thus we obtain that
$$ A(*)\simeq K(S) \simeq \K(\mathcal{E}) \simeq
\mathbb{Z} \times |BGL(\mathcal{E})|^+ \,.$$
\item
For a %%%170809
commutative
%%%170809
ring $A$ we consider the following small
\bc of finitely generated free $A$-modules,
$\mathcal{F}(A)$. Objects of $\mathcal{F}(A)$  are the finitely
generated  free $A$-modules
$A^n$ for $n \geq 0$. The set of morphisms from $A^n$ to $A^m$ is empty
unless $n=m$, and the morphisms from $A^n$ to itself are the
$A$-module automorphisms of $A^n$, \ie, $GL_n(A)$. Our result
allows us to identify the two-fold iterated algebraic $K$-theory of $A$,
%% JR3-1: inserted K(K(A))
$K(K(A))$,
with $\mathbb{Z} \times |BGL(\mathcal{F}(A))|^+$.
\item
The case that started the project is the category of $2$-vector spaces
of Kapranov
and Voevodsky \cite{KV}, viewed as modules over the \bc $\mathcal{V}$
of complex (Hermitian) vector spaces.  Here $\mathcal{V}$
has objects $\mathbb{C}^n$ for $n \geq 0$, and the
automorphism space of $\mathbb{C}^n$ is the unitary group $U(n)$.
%% JR3-1: I'd like K(ku) = \K(\mathcal{V}) to be more visible
This identifies $K(H\mathcal{V}) = K(ku)$ with $\K(\mathcal{V})
\simeq \mathbb{Z} \times |BGL(\mathcal{V})|^+$,
which was called the $K$-theory of the $2$-category of complex $2$-vector
spaces in \cite{BDR}.
%%%300107
Ausoni's calculations \cite{A} show that
$K(ku_p)$ has telescopic complexity $2$ for every prime
$\geq 5$, and  thus  qualifies as a form of elliptic cohomology.
%%%%
\item
Replacing the complex numbers by the reals yields an identification of
$K(ko)$ with the $K$-theory of the $2$-category of real $2$-vector
spaces.
\item
Considering other subgroups of $GL_n(\mathbb{C})$ or
$GL_n(\mathbb{R})$ as morphisms in %jr1804
a category with objects $\bfn =
\{1,\ldots,n\}$ with  $n \geq 0$ gives a large variety of $K$-theory
spectra that are in the range of our result. For a sample of such
species have a look at \cite[pp.~161--167]{May2}.
\end{itemize}

%% JR3-1: Some conjecture about $N$-fold iterated $K$-theory here?

\subsection{The spine of the argument giving a proof of Theorem \ref{thm:main}}
%The contents of this paper is a proof of the main Theorem \ref{thm:main}.
Although the proof contains some lengthy technical lemmas, it is
possible to  give the main flow of the argument in a few paragraphs,
referring  away the hard parts.

Remember that the group-like monoid $GL_n(H\R)$ is defined by the pullback
$$
\xymatrix{
{GL_n(H\R)} \ar[rr] \ar[d] & & {\hocolim_{\bfm \in I} \Omega^m M_n(H\R
  (S^m))}  \ar[d]\\
{GL_n(\pi_0 H\R)} \ar[rr] & & {M_n(\pi_0 H\R) \rlap\,.}
}
$$
where $I$ is %%%the
a strictly monoidal skeleton of the category of
finite sets  and injective functions, making $GL_n(H\R)$ a group-like
%% JR3-1: Changed on to in
monoid  (here we have written $H\R$ in the form of a simplicial
functor with  associated symmetric spectrum $m\mapsto H\R(S^m)$).

%%%br140411
Let $M_n(\bar{\R})$ be the monoidal category of $n\times n$-matrices
over $\bar{\R}$ (see Section \ref{sec:invertiblemat}).  The set of %jr1804
components $\pi_0M_n(\bar{\R})$ can be 
identified with $M_n(\pi_0\bar{\R})$, %jr1804
and we let $GL_n(\bar{\R})$ be the submonoidal 
category of $M_n(\bar{\R})$ consisting of the components that are invertible
as matrices over the additive group completion of $\pi_0\bar{\R}$ (see
Definitions~\ref{def:wimatrices} and \ref{def:glnr}).  Since
the effect of stabilization 
in $H\bar{\R}$ is exactly group completion, the natural isomorphism 
$|\bar{\R}| \ra H\bar{\R}(S^0)$ together with stabilization induces a
natural map $|GL_n(\bar{\R})| \ra GL_n(H\bar{\R})$.
%%%br140411

\begin{lem}\label{lem:glh}
If $\bar{\R}$ is a  ring category, \ie, a rig category with
%% JR3-1: Removed extra ()s
$\pi_0 \bar{\R}$ a ring, then
  $$|GL_n(\bar{\R})| \we GL_n(H\bar{\R})$$
is a homotopy equivalence.
\end{lem}

%%%br140411
% Here, $GL_n(\bar{\R})$ denotes the components of the monoidal category of
% $n\times n$-matrices over $\bar{\R}$, $M_n(\bar{\R})$,  corresponding
% %% JR3-1: Removed extra ()s
% to $GL_n(\pi_0 \bar{\R})$ (see Definition \ref{def:glnr}).  

\begin{proof}
%% JR3-1: Removed extra ()s
By assumption $\pi_0 \bar{\R}$
is isomorphic to its group completion $Gr(\pi_0 \bar{\R}) = \pi_0H \bar{\R}$,
%% JR3-1: Added ()s
so it is enough to show that $|M_n(\bar{\R})|$ and  $\hocolim_{\bfm \in I}
\,\Omega^mM_n(H\bar{\R}(S^m))$ are
equivalent.  Both are $n^2$-fold products, so it suffices to show that
$|\bar{\R}|$ and $\hocolim_{\bfm \in I} \Omega^mH\bar{\R}(S^m)$ are
equivalent.
All the structure maps
$\Omega^mH\bar{\R}(S^m)\to \Omega^{m'} H\bar{\R}(S^{m'})$ are
equivalences for  $\bfm
\to  \mathbf{m'}$ an
injection of nonempty finite sets, and since $\bar{\R}$ is already group-like
$|\bar{\R}| \simeq H\bar{\R}(S^0)$ maps by an equivalence to $\Omega
H\bar{\R}(S^1)$.
\end{proof}

%%%020909When $\R$ satisfies the conditions of Theorem \ref{thm:main}, we

We know from  \cite{BDRR1} that there is a chain of %%%br1004 added simplicial 
simplicial rig categories
$$
\begin{CD}
  \R@<{\sim}<<Z\R@>>>\bar\R
\end{CD}
$$
such that
\begin{itemize}
\item $\R\gets Z\R$ becomes a weak equivalence upon realization,
\item $HZ\R \to H\bar\R$ is a stable equivalence and
\item $\pi_0\bar\R$ is a ring.
\end{itemize}
Consider the commutative diagram
$$
\xymatrix{
{|BGL(\R)|} \ar[d] & {|BGL(Z\R)|}\ar[r]^{} \ar[l]_{\sim}\ar[d] &
{|BGL(\bar{\R})|}\ar[d]^{\sim} \\
{BGL(H\R)}& {BGL(HZ\R)} \ar[l]_{\sim} \ar[r]^{\sim}&{BGL(H\bar{\R})}
%% JR3-1: Moved comma
\rlap\,,
}$$
where the definition of $GL$ of rig categories is %170809 defined
given in Section
\ref{sec:invertiblemat} and the bar construction is recalled in
Section
\ref{sec:barc}.  Both constructions preserve weak %jr1804
equivalences.
\begin{itemize}
\item The leftward %jr1804
pointing horizontal maps are weak equivalences since
$\R\gets Z\R$ is,
\item the bottom rightward %jr1804
pointing arrow is a weak equivalence since
$HZ\R\to H\bar\R$ is a stable equivalence and
\item the right hand %jr1804
vertical map is a weak equivalence by Lemma \ref{lem:glh}.
\end{itemize}

%% JR3-1: Rewording and reordering wrt the itemlist above
To show that the left hand %jr1804
vertical arrow is a weak equivalence, it
therefore suffices to prove that the upper right hand horizontal
map is a weak equivalence.
By Proposition \ref{prop:onesided is fiber}, the homotopy fiber of
$BGL(Z\R)\to BGL(\bar\R)$ is given by the one-sided bar construction
$B(*,GL(Z\R),GL(\bar\R))$, and so we have reduced the problem to
giving a contraction of the associated space.

%%%140809
Under the assumptions of Theorem \ref{thm:main} we know
%%%
from \cite{BDRR1} that there is a %%%140809we also have a
chain of weak equivalences
$$
\begin{CD}
  (-\R)\R@<{\sim}<<Z(-\R)\R@>{\sim}>>\bar\R
\end{CD}
$$
of $Z\R$-modules. %%%140809 inserted. Here,
Here, ($-\R)\R$ is the Grayson-Quillen model \cite{G}, also
discussed in Section \ref{sub:TM}, and so by Remark \ref{GLofmodules}
we get weak equivalences
$$
\begin{CD}
  B(*,GL(\R),GL((-\R)\R))@<{\sim}<< B(*,GL(Z\R),GL(Z(-\R)\R))
@>{\sim}>> B(*,GL(Z\R),GL(\bar\R)) \,.
\end{CD}
$$
 This means that
the somewhat complicated construction of $\bar\R$ from \cite{BDRR1}
may be safely forgotten once we know it exists.

Just one simplification remains: in Lemma \ref{lem:TM is gp compl} we
display  a weak equivalence $T\R\to (-\R)\R$ of $\R$-modules.
The $\R$-modules $T\R$ and $(-\R)\R$ are generalizations of how
%%%140809 you make
to construct the  integers from the natural numbers by considering equivalence
classes of pairs of natural numbers.

We are left with showing that $B(*,GL(\R),GL(T\R))$ is contractible.
This is  done through a concrete contraction.  It is an elaboration of
the  following path
$$
\left[
  \begin{matrix}
    a-b&0\\0&1
  \end{matrix}
\right]
\to
\left[
  \begin{matrix}
    a-b&b\\0&1
  \end{matrix}
\right]
\gets
\left[
  \begin{matrix}
    1&0\\-1&1
  \end{matrix}
\right],\qquad a,b\in\mathbb N
$$
 in $B(*,GL_{2}(\mathbb N),GL_{2}(\mathbb Z))=\left\{[q]\mapsto
   GL_{2}(\mathbb N)^{\times q}\times GL_{2}(\mathbb Z)\right\}$, with $1$-simplices given by $$\left(\left[
  \begin{matrix}
    1&b\\0&1
  \end{matrix}
\right],\left[
  \begin{matrix}
    a-b&0\\0&1
  \end{matrix}
\right]\right)\text{ and }\left(\left[
  \begin{matrix}
    a&b\\1&1
  \end{matrix}
\right],\left[
  \begin{matrix}
    1&0\\-1&1
  \end{matrix}
\right]\right)\qquad \in GL_2(\mathbb N)\times GL_2(\mathbb Z)$$
%multiplication by the matrices with entries in $\mathbb N$
% $$
% \left[
%   \begin{matrix}
%     1&b\\0&1
%   \end{matrix}
% \right]\text{ and }\left[
%   \begin{matrix}
%     a&b\\0&1
%   \end{matrix}
% \right]
% $$
%% JR3-1: Reworded connected image to null-homotopic inclusion
%% JR3-1: which is a little less than obvious
respectively, showing %(the obvious fact)
that the inclusion of
$B(*,GL_{1}(\mathbb N),GL_{1}(\mathbb Z))$
in $B(*,GL_{2}(\mathbb N),GL_{2}(\mathbb Z))$ is null-homo\-topic.

\subsection{Plan} %%%140809 section --> Section
The structure of the paper is as follows.
In Section \ref{sec:invertiblemat} we discuss the monoidal category
$GL_n(\R)$ of (weakly) invertible matrices over a strictly bimonoidal category
$\R$.
Section \ref{sec:barc} recalls the definition of the bar construction of
monoidal categories as in \cite{BDR} and introduces a version with
coefficients in a module. In Section \ref{sec:contraction} we
construct the  promised contraction of $B(*,GL(\R),GL(T\R))$, thus
completing  the proof of the main theorem.

\bigskip

This paper has circulated in preprint form under the title
``Two-vector bundles define a form of elliptic cohomology'', which,
while not highlighting the nuts and bolts of the paper perhaps better
represented the reason for writing (or reading) it.  The old preprint
\cite{BDRR}
also contained the main result of the paper \cite{BDRR1}. %jr1804

%% JR3-1: Here is BR's suggestion:
Graeme Segal constructed a ring completion of the rig category of
complex vector spaces \cite[p.~300]{S:catscoh} (see also the Appendix of
\cite{S:equ}). His model is a topological category consisting of certain
spaces of bounded chain complexes and spaces of quasi-isomorphisms. One
can probably build a variant of his model that could %jr1840
replace the construction $\bar\R$ in
our proof of Theorem~\ref{thm:main}, in the special case $\R = \mathcal{V}$.

A piece of notation: if $\calc$ is any small category, then
the expression $X \in \calc$ is short for ``$X$ is an object of
$\calc$'' and likewise for morphisms and diagrams.  Displayed diagrams
commute unless the contrary is stated explicitly.
For basics on bipermutative and \bcs we refer to \cite{BDRR1} Section 2.

\section{Weakly invertible matrices} \label{sec:invertiblemat}

Let $\R$ be a \sbc.
%%JR3-1: Can we give a reference for the definition of a \sbc?

\begin{defn}
The \emph{category of $n \times n$-matrices over $\R$}, $M_n(\R)$, is
defined
as follows. The objects of $M_n(\R)$ are matrices
$X = {(X_{i,j})}_{i,j=1}^n$ of objects of $\R$ and morphisms from $X =
{(X_{i,j})}_{i,j=1}^n$ to $Y = {(Y_{i,j})}_{i,j=1}^n$ are
matrices  $F = {(F_{i,j})}_{i,j=1}^n$ where each $F_{i,j}$ is a
morphism in $\R$ from $X_{i,j}$ to $Y_{i,j}$.
\end{defn}
\begin{lem} \label{lem:monoidalmat}
For a \sbc $(\R, \oplus, 0_\R, c_\oplus, \otimes, 1_\R)$  the category
$M_n(\R)$ is a monoidal
category  with respect to the matrix multiplication bifunctor
\begin{align*}
M_n(\R) \times M_n(\R) & \stackrel{\cdot}{\lra}  M_n(\R)\\
{(X_{i,j})}_{i,j=1}^n \cdot  {(Y_{i,j})}_{i,j=1}^n & =
{(Z_{i,j})}_{i,j=1}^n \quad \text{with} \quad Z_{i,j} = \bigoplus_{k=1}^n
X_{i,k} \otimes Y_{k,j} \,.
\end{align*}
The unit of this structure is given by the unit matrix object $E_n$ which has
$1_\R \in \R$ as diagonal entries and $0_\R \in \R$
in the other places.
\end{lem}
%% JR3-1: Removed extra ()s
The property of $\R$ being bimonoidal gives $\pi_0 \R$ the structure
of a  rig, and its (additive) group completion
$Gr(\pi_0 \R) = (-\pi_0\R)\pi_0\R$ is a ring.
\begin{defn} \label{def:wimatrices}
%% JR3-1: Removed extra ()s
We define the \emph{weakly invertible $n \times n$-matrices over $\pi_0 \R$},
$GL_n(\pi_0 \R)$, to be the $n \times n$-matrices over $\pi_0 \R$
that are invertible as matrices over $Gr(\pi_0 \R)$.
\end{defn}
Note that we can define $GL_n(\pi_0 \R)$ by the pullback square
$$\xymatrix{
{GL_n(\pi_0\R)} \ar[r] \ar@{ >->}[d] & {GL_n(Gr(\pi_0\R))} \ar@{ >->}[d] \\
{M_n(\pi_0\R)} \ar[r] & {M_n(Gr(\pi_0\R))}
}$$
\begin{defn} \label{def:glnr}
The \emph{category of weakly invertible  $n \times n$-matrices over $\R$},
$GL_n(\R)$, is the full subcategory of $M_n(\R)$ with objects all
matrices $X = {(X_{i,j})}_{i,j=1}^n \in M_n(\R)$ whose matrix of
$\pi_0$-classes $[X] = {([X_{i,j}])}_{i,j=1}^n$ is contained in
$GL_n(\pi_0 \R)$.
\end{defn}
Matrix multiplication is of course compatible with the property of
being weakly invertible. Thus, the category $GL_n(\R)$ inherits a
monoidal structure from $M_n(\R)$.

However, even if our base category is not bimonoidal it still makes
sense to talk about matrices and even weakly invertible matrices, as long as
$\pi_0$ of that category is a rig.

%{\color{blue}
  \begin{rem}\label{GLofmodules}If $\M$ is an $\R$-module, matrix
multiplication makes the category $M_n(\M)$ into a module over the
monoidal  category $GL_n(\R)$.
%% JR3-1: The following is a bit cryptic.  In the noncommutative case
%% JR3-1: it is important that \pi_0R is central in \pi_0M and \pi_0N.
%% JR3-1: This is OK in the application to (-R)R, Z(-R)R, \bar R.
For our applications the following situation will be particularly
important:  let $\M\to\mathcal N$ be a map of $\R$-modules, where
the map of $\pi_0\R$-modules $\pi_0\M\to\pi_0\mathcal N$ comes from
a rig map under $\pi_0\R$.
Then we get a map $GL_n(\M)\to GL_n(\mathcal
N)$ of  $GL_n(\R)$-modules which induces a weak equivalence upon
realization  if $\M\to\mathcal N$ does.
  \end{rem}

There is a canonical stabilization functor $GL_n(\R) \ra GL_{n+1}(\R)$
which is induced by taking the block sum with $E_1 \in GL_1(\R)$. Let
%% JR3-1: Inserted sequential, as opposed to over I
$GL(\R)$ be the sequential colimit of the categories $GL_n(\R)$.

\section{The one-sided bar construction} \label{sec:barc}
%% JR3-1: Switched the roles of \mathcal M and \mathcal T in this section,
%% JR3-1: so that T matches T(R) later, and group completion is M \to G
In this section we review some well-known facts about the one-sided
bar  construction of monoidal categories.
\begin{defn}
  Let $(\mathcal M, \cdot, 1)$ be a monoidal category and
$\mathcal T$ a left $\mathcal
  M$-module.  The \emph{one-sided bar construction} $B(*,\mathcal M,\mathcal
  T)$ is the simplicial category whose %%%br1004
  category of $q$-simplices $B_q(*,\mathcal
  M,\mathcal T)$ is as follows: 
%%% JR3-1: I prefer is to are here
%%%  is the following category: 
%%%br1004
consider the ordered set
  $[q]_+=[q]\sqcup \{\infty\}$, \ie, in addition to the numbers
$0<1<\dots<q$
%% JR3-1: Strengthened maximal to greatest
there is a greatest element $\infty$.  An object  $a$ in
$B_q(*,\mathcal
  M,\mathcal T)$ consists of the following data.
\begin{enumerate}
  \item For each $0\leq i<j\leq q$
  there is an object $a_{ij}\in\mathcal M$, and for each $0\leq i\leq
  q$ an object $a_{i\infty}\in\mathcal T$.
\item For each $0\leq i<j<k\leq\infty$ there is an isomorphism
$$a_{ijk}\colon a_{ij}\cdot a_{jk}\to a_{ik}$$
(in $\mathcal M$ if $k<\infty$ and in $\mathcal T$ if $k=\infty$)
such that if  $0\leq i<j<k<l\leq\infty$, the following diagram commutes
$$\xymatrix{
{(a_{ij}\cdot a_{jk})\cdot a_{kl}} \ar[d]_{a_{ijk}\cdot
    \id}\ar[rr]^{\text{struct.~iso.}} & &
{a_{ij}\cdot (a_{jk}\cdot a_{kl})} \ar[d]^{\id\cdot a_{jkl}}\\
{a_{ik}\cdot a_{kl}} \ar[r]^{a_{ikl}} & {a_{il}}& {a_{ij}\cdot
a_{jl} \rlap\,.} \ar[l]_{a_{ijl}}
}
$$
  \end{enumerate}
A morphism $f\colon a\to b$ consists of morphisms $f_{ij}\colon
a_{ij}\to b_{ij}$ (in $\mathcal M$ if $j<\infty$ and in $\mathcal T$
if $j=\infty$) such that if $0\leq i<j<k\leq\infty$
$$f_{ik}a_{ijk}=b_{ijk}(f_{ij}\cdot f_{jk})\colon a_{ij}\cdot
a_{jk}\to b_{ik} \,.$$

The simplicial structure is gotten as follows: if $\phi\colon
[q]\to[p]\in\Delta$ the functor
$\phi^*\colon B_p(*,\mathcal M,\mathcal T)\to
B_q(*,\mathcal M,\mathcal T)$ is obtained by precomposing with
$\phi_+=\phi\sqcup\{\infty\}$. So for instance $d_1(a)$ is gotten by
deleting all entries with indices containing
$1$ from the data giving $a$. In order to allow for degeneracy maps
$s_i$, we use the convention that  all objects of the form
$a_{ii}$ are the unit of the monoidal structure, and all isomorphisms
of the form $a_{iik}$ and $a_{ikk}$ are identities.
\end{defn}

\begin{rem} %jr1804
A good way to think about this comes from the discrete case %jr1804
when $\mathcal M$ is a monoid and $\mathcal T$ is %%%br1004 a 
an $\mathcal M$-set.  Then an
object  $a\in B_q(*,\mathcal M,\mathcal T)$ is uniquely given by the
``superdiagonal'' %jr1804
$(a_{01},a_{12},\dots, a_{q-1\, q},a_{q\infty})$, and
$B(*,\mathcal M,\mathcal T)$ is isomorphic to the nerve of the category with
%% JR3-1: Changed objects from T to M, or vice versa
objects  $\mathcal T$ and
morphisms $a_{1\infty}\to a_{01}\cdot a_{1\infty}=a_{0\infty}$
corresponding to $(a_{01},a_{1\infty})$.

%%%br140411
The reason we have to include all of %jr1804
the ``upper triangular'' %jr1804
elements is that
associativity may not be strict. For instance, it is not strict in our
main example: matrix multiplication over a rig category is in general
not strictly associative.  Hence, $a_{01}\cdot (a_{12} \cdot a_{23})$
may be different  from $(a_{01} \cdot a_{12}) \cdot a_{23}$, and so
the superdiagonal %jr1804
elements do not carry 
enough information to turn the obvious choice of face maps into a
simplicial structure.  We remedy this by adding choices for all faces
in our simplices.  This just adds more elements in each isomorphism
class in every simplicial degree, and is a standard trick used in many
places, for instance by Waldhausen in his $S_\bullet$-construction.  %jr1804
%%%br140411
\end{rem}

\begin{ex}

  \begin{enumerate}
  \item If $\mathcal T$ is the one-point category $*$, then $B(*, \mathcal M,
    *)$ is isomorphic to the bar construction $B\mathcal M$ of \cite{BDR}.
  \item If $F\colon\mathcal M\to\mathcal M'$ is a lax monoidal functor, then
    $\mathcal M'$ may be considered as an $\mathcal M$-module, and we
    write without further ado $B(*,\mathcal M,\mathcal M')$ for the
    corresponding bar construction (with $F$ suppressed).  In case $F$
    is an isomorphism, $B(*,\mathcal M,\mathcal M')$ is contractible.
  \end{enumerate}

\end{ex}
We think of elements of $B_q(*,\mathcal M,\mathcal T)$ in terms of
strictly upper triangular %jr1804
arrays of objects, suppressing the isomorphisms, so that a
typical element in $B_2(*,\mathcal M,\mathcal T)$ is written
$$
\begin{matrix}
  a_{01}&a_{02}&a_{0\infty}\\
        &a_{12}&a_{1\infty}\\
        &      &a_{2\infty}
\end{matrix}
$$
with $d_1$ given by
$$
\begin{matrix}
  a_{02}&a_{0\infty}\\
        &a_{2\infty} \rlap\,.
\end{matrix}
$$
The one-sided bar construction is functorial in ``natural modules''.
A natural module is a pair $(\mathcal M,\mathcal T)$ where $\mathcal
M$ is a monoidal category and $\mathcal T$ is an $\mathcal M$-module.
A morphism $(\mathcal M,\mathcal T)\to(\mathcal M',\mathcal T')$
consists of a pair $(F,G)$ where $F\colon \mathcal M\to\mathcal M'$ is
a lax monoidal functor and $G\colon\mathcal T\to F^*\mathcal T'$ is a
map of $\mathcal M$-modules, where $F^*\mathcal T'$ is $\mathcal T'$
endowed with the $\mathcal M$-module structure given by restricting
along $F$.

\begin{lem}\label{lem:diagmatters}
  For each $q$ there is an equivalence of categories between
  $B_q(*,\mathcal M,\mathcal T)$ and the product category $\mathcal
  M^{\times q}\times\mathcal T$.
\end{lem}
\begin{proof}
  The equivalence is given by the forgetful functor
$$F\colon B_q(*,\mathcal M,\mathcal T)\to\mathcal M^{\times
  q}\times\mathcal T$$
sending $a$ to the ``superdiagonal''
$F(a)=(a_{01},\dots,a_{q-1\,q},a_{q\infty})$.   The
inverse is gotten by sending $(a_1,\dots,a_q,a_\infty)$ to the $a$
with $a_{ij}=a_{i+1}\cdot(\cdots (a_{j-1}\cdot a_j)\cdots)$ and
$a_{ijk}$ given by the structural isomorphisms.
\end{proof}
\begin{cor}\label{cor:Bpreserves equivalences}
  Let $(F,G)\colon(\mathcal M,\mathcal T)\to(\mathcal M',\mathcal T')$
  be a map of natural modules such that $F$ and $G$ are equivalences
  of categories.  Then the induced map
$$B(*,F,G)\colon B(*,\mathcal M,\mathcal T)\to B(*,\mathcal
M',\mathcal T')$$
is a degreewise equivalence of simplicial categories.
\end{cor}

%{\color{blue}
\begin{rem}\label{rem:Bpreserves we}
  The same result holds, if instead of equivalences of categories we
  consider  weak equivalences.
\end{rem}
%}

Usually $\mathcal M^{\times q}\times\mathcal T$ is not functorial in
$[q]$, but if $(\mathcal M,\mathcal T)$ is strict, the monoidal
structure gives a simplicial category
$$B^{\text{strict}}(*,\mathcal M,\mathcal T)=
\{[q]\mapsto \mathcal M^{\times q}\times\mathcal T\} \,.$$
In this situation Lemma \ref{lem:diagmatters} reads:
\begin{cor}\label{cor:Binstrictcase}
  Let $\mathcal M$ be a strict monoidal category and $\mathcal T$ a
  strict $\mathcal M$-module.  Then there is a degreewise equivalence
  between the simplicial categories
  $B(*,\mathcal M,\mathcal T)$ and $B^{\text{strict}}(*,\mathcal M,\mathcal T)$.
\end{cor}

\begin{prop}\label{prop:onesided is fiber}
  Let $F\colon\mathcal M\to\mathcal G$ be a strong monoidal functor
  such that the monoidal structure on $\mathcal G$ induces a group
  structure on $\pi_0\mathcal G$.  Then
$$
\begin{CD}
  B(*,\mathcal M,\mathcal G)@>>>B\mathcal M\\
@VVV@VVV\\
B(*,\mathcal G,\mathcal G)@>>>B\mathcal G
\end{CD}
$$
is homotopy cartesian, meaning that it induces a homotopy cartesian
diagram upon applying the nerve functor in every degree.
The (nerve of the) lower left hand corner is contractible.
\end{prop}
\begin{proof}
  By \cite{JoyalStreet} there is a diagram of monoidal categories
$$
\begin{CD}
  \mathbf{st}\mathcal M@>\sim>>\mathcal M\\
@V{\mathbf{st}F}VV@V{F}VV\\
  \mathbf{st}\mathcal G@>\sim>>\mathcal G
\end{CD}
$$
such that the horizontal maps are monoidal equivalences, and
$\mathbf{st}F$ is a strict monoidal functor between strict monoidal
categories.  Together with Corollaries \ref{cor:Bpreserves equivalences}
and \ref{cor:Binstrictcase}
this tells us that we may just as well consider the strict situation,
and use the strict bar construction.  However, note that the nerve of
the strict monoidal category $\mathbf{st}\mathcal M$ is a simplicial
monoid, and that reversal of priorities gives a natural isomorphism
$$B(*,N\mathbf{st}\mathcal M,N\mathbf{st}\mathcal G)\cong
NB^{\text{strict}}(*,\mathbf{st}\mathcal M,\mathbf{st}\mathcal G),$$
so that our statement reduces to the statement that
$$B(*,N\mathbf{st}\mathcal M,N\mathbf{st}\mathcal G)\to
B(N\mathbf{st}\mathcal M)\to B(N\mathbf{st}\mathcal G)$$
is a fiber sequence up to homotopy, which is a classical result
\cite{May1}
%% JR3-1: More precisely, if the reference is [May1] as added below,
%% JR3-1: I think we only get a quasi-fibration.
given that $N\mathbf{st}\mathcal G$ is group-like.
\end{proof}
%(Note: the ``strong'' condition comes from the reference to
%\cite{JoyalStreet}, and should not be necessary)

\section{Contracting the one-sided bar construction} \label{sec:contraction}

\subsection{A model for $K$-theory of $\R$ as an $\R$-module}
\label{sub:TM}
%%%170809 4.1 rewritten BD,BR

In order to construct concrete homotopies, we offer a slight variant
of the Grayson--Quillen model where morphisms are not entire equivalence
classes.  The price is as usual that the resulting object is a
$2$-category. %jr1804
Since there was some confusion about this point 
while the paper was still at a preprint stage, we emphasize
that this is not the construction of Thomason \cite[4.3.2]{Th1} and
Jardine \cite{J}.

%% JR3-1: I'd like this to be a numbered definition of TM
\begin{defn}
Let $(\M, \oplus,0_\M, \tau_\M)$ be a permutative category written
additively.  Let
$T\M$ be the
following $2$-category.  The objects of $T\M$ are pairs $(A^+,A^-)=:A$ of
objects  in $\M$, thought of as plus and minus objects in $\M$.  Given
two  objects $A,B \in T\M$, the category of morphisms $T\M(A,B)$ has
objects the pairs  $(X,\alpha)$ where $X$ is an object in $\M$ and $\alpha$
is a  pair of morphisms $\alpha^\pm \colon
A^\pm \oplus X \to B^\pm$ in $\M$.
A morphism from
$(X,\alpha)$ to $(Y,\beta)$ is an isomorphism $\phi\colon X \to Y$ such that 
%%%br1004$\beta(1 \oplus \phi)=\alpha$. 
$\beta^\pm(1 \oplus \phi)=\alpha^\pm$. 
Composition $T\M(B,C)\times T\M(A,B)\to
T\M(A,C)$ is  given by sending  $((Y,\beta), (X,\alpha))$ to
the pair
consisting  of $X\oplus Y$ and the composite maps
$$
\begin{CD}
  A^\pm \oplus (X \oplus Y) =
(A^\pm \oplus X) \oplus Y @>{\alpha^\pm \oplus \id}>>
B^\pm \oplus Y@>\beta^\pm>>C^\pm \,.
\end{CD}
$$
Composition on
morphisms  is simply given by addition. Composition is strictly
%%%additive
associative because $\M$ is permutative; if $\M$ is merely symmetric
monoidal,  standard modifications are necessary.  Symmetry
allows for a symmetric monoidal structure on
$T\M$: if we define $(A^+,A^-) \oplus (B^+,B^-) := (A^+ \oplus B^+,A^-
\oplus B^-)$, we need the symmetry in order to turn that prescription
into a bifunctor.
\end{defn}

The Grayson-Quillen model for the $K$-theory of $\M$ is the category
$(-\M)\M$ with the same objects as $T\M$ and with morphism sets the
path components $(-\M)\M(A,B)=\pi_0T\M(A,B)$.
If all morphisms in $\M$ are isomorphisms and if additive
translation in $\M$ is faithful,  $(-\M)\M$ is shown in \cite{G} to be
a group completion of $\M$.  Under these hypotheses, there is at most
one morphism between two given objects $(X,\alpha)$ and  $(Y,\beta)$
in $T\M(A,B)$.  Consequently the morphism spaces are homotopy
discrete: the projection $T\M(A,B)\to \pi_0T\M(A,B)$ is a weak
equivalence.  In the case of a topological category,
we interpret $\pi_0$ as the coequalizer of the source and target maps
from the morphism space to the object space.
The assignment that is the
identity  on objects and otherwise is induced by the projection
$T\M(A,B)\to \pi_0T\M(A,B)$ gives a $2$-functor $T\M\to (-\M)\M$
(considering $(-\M)\M$ as a $2$-category with only identity
$2$-morphism).

\begin{lem}\label{lem:TM is gp compl}
  Let $\M$ be a permutative category with all morphisms in $\M$
isomorphisms and faithful additive
translation.  The $2$-functor $T\M \to
(-\M)\M$ is a weak  equivalence and the standard inclusion $\M \to T\M$ is a
group  completion.
\end{lem}

Note that if $\R$ is a \bc, $T\R$ will not be a \bc
(essentially because of the non-strict symmetry
in  quadratic terms, as in \cite[p.~572]{Th2}),
but it will still be an $\R$-module:
\begin{lem} \label{lem:R-module}
  Let $(\R, \oplus, 0_\R, c_\oplus, \otimes, 1_\R)$ be a
strictly bimonoidal category.  The map
$$\R\times T\R\to T\R$$
given on objects by $(A,(B^+,B^-))\mapsto(A \otimes B^+,A \otimes
B^-)$, and
on  morphisms by sending $\phi\colon A \to B \in \R$ and $(X,\alpha)\in
T\R(C,D)$  to the pair consisting of $A \otimes X$ and the map
$$
\begin{CD}
  A\otimes C^\pm \oplus A \otimes X @>{}>>
A \otimes (C^\pm \oplus X)@>{\phi\otimes \alpha^\pm}>>B\otimes D^\pm
\end{CD}
$$
(where the first map is the left distributivity isomorphism) induces an
$\R$-module structure on $T\R$.
\end{lem}

%%%simplicial version
We consider  $T\R$ as a simplicial
category by taking the nerve of each category of morphisms; thus
in simplicial degree $\ell$, the objects of
$T_\ell\R$ are the objects of $T\R$. The morphisms in $T_\ell\R$ from
$(A^+,A^-)$
to $(B^+,B^-)$ consist of objects $X^0,\ldots,X^\ell$, a $1$-morphism %jr1804
$\alpha^\pm \colon A^\pm \oplus X^0 \rightarrow B^\pm$, and isomorphisms
%%%br1004 changed l to r
$\phi^r\colon X^r \rightarrow X^{r-1}$ for $r=1, \ldots, \ell$.
%% JR3-1: Possible clarification
%%%Note that $l$ and $\ell$ are different symbols.
The simplicial structure is
given by composing and forgetting $\phi^r$'s and inserting identity
maps.

\subsection{Subdivisions}
We will use the following variant of edgewise subdivision to make
room for an explicit simplicial contraction, whose construction
begins in Subsection \ref{subsec:start}.
Consider the shear functor $z\colon \Delta\times\Delta\to
\Delta\times\Delta$  given by sending $(S,T)$ to $(T\sqcup S,T)$ where
$T\sqcup S$ is the disjoint union with the ordering obtained from $T$
and $S$  with the extra declaration that every object in $S$ is
greater than  every object in $T$.  If $B$ is a bisimplicial object, we
let  $z^*B=B \circ z$.  The standard inclusion $S\to T\sqcup S$ gives a
natural  transformation $\eta$ in $\Delta\times \Delta$ from the
identity to $z$, and hence a natural  transformation in bisimplicial sets
$\eta^*\colon z^*\to \id$. Let $\mathrm{Ens}$ denote the category of sets
and functions.

\begin{lem}
  For any bisimplicial set $X$ the map $\eta^* \colon z^*X\to X$
  becomes a weak equivalence upon realization.
\end{lem}
\begin{proof}
  The diagonal of $z^*X$ is equal to the %%%3110composite
evaluation of $X$ on the opposite of the composite
$$
\begin{CD}
  \Delta@>{S\mapsto
(S,S)}>>\Delta \times\Delta @>{(S,T)\mapsto
    (S\sqcup S,T)}>>\Delta\times\Delta \,, %@>{X}>>\mathrm{Ens},
\end{CD}
$$
so since a map of bisimplicial sets is an equivalence if it is one in
every (vertical) degree, it is enough to know that for each fixed $T\in
\Delta$ the natural map $\{S\mapsto X(S\sqcup S,T)\}\to \{S\mapsto
X(S,T)\}$ is a weak equivalence.  But this is a standard weak
equivalence from the (second) edgewise subdivision, which is known to
be homotopic to a homeomorphism after realization.
See \cite[Lemma 1.1]{BHM} and the proof of \cite[Proposition 2.5]{BHM}.
\end{proof}

Vertices in $z^*(\Delta[p]\times\Delta[q])$ (where products of
simplicial sets are viewed as bisimplicial sets,
and vertices are $(0,0)$-simplices) are for instance indexed
by tuples  $((a, b), c)$ where $0\leq a \leq b \leq p$ and $0\leq c\leq q$.

Here are pictures of $z^*(\Delta[2]\times\Delta[0])$ and
$z^*(\Delta[2]\times \Delta[1])$:
$$\xymatrix{
((2,2),0)\ar[r]\ar[dr]&((1,2),0)\ar[d]\ar[r]&((1,1),0)\ar[d]\\
&((0,2),0)\ar[r]\ar[dr]&((0,1),0)\ar[d]\\
&&((0,0),0)}$$

$$\xymatrix{
((2,2),1)\ar[rr]\ar[dd]\ar[drrr]\ar[ddrr]\ar@/^1.0pc/[dddrrr] &&
((1,2),1)\ar[dd]|(0.29)\hole|(0.70)\hole \ar[dr]\ar[rr]\ar[dddr]|(0.23)\hole &&
((1,1),1)\ar[dd]|(0.5)\hole|(0.66)\hole \ar[dr]
	\ar[dddr]|(0.33)\hole|(0.5)\hole \\
&&& ((0,2),1)\ar[drrr]\ar[rr]\ar[dd] &&
((0,1),1)\ar[dr]\ar[dd]|(0.33)\hole \\
((2,2),0)\ar[rr]\ar[drrr] &&
((1,2),0)\ar[dr]\ar[rr]|(0.21)\hole|(0.33)\hole|(0.5)\hole &&
((1,1),0)\ar[dr] &&
((0,0),1) \ar[dd] \\
&&& ((0,2),0)\ar[drrr]\ar[rr]
&& ((0,1),0)\ar[dr] \\
&&&&&& ((0,0),0) \,.}$$

%((Prove that $z^*B\to B\to z^*B$ is homotopic to the identity.  Not
%really  needed for the conclusions we want))

\noindent
Note that for any bisimplicial set $X$
$$ X_{(p,q)} \cong \text{Hom}_{\text{bisimp. sets}}
(\Delta[p]\times\Delta[q],X)
= \int_{([s],[t])}
\mathrm{Ens}(\Delta([s],[p])\times\Delta([t],[q]),X_{(s,t)})$$
as a categorical end.
Therefore,
%The associated extension is
the right adjoint of $z^*$, $z_*$, is given by
$$ (z_*X)_{(p,q)} = \text{Hom}_{\text{bisimp. sets}}
(\Delta[p]\times\Delta[q],z_*X) \cong  \text{Hom}_{\text{bisimp. sets}}
(z^*(\Delta[p]\times\Delta[q]),X)$$
and thus
\begin{align*}
  z_*X&=\{[p],[q]\mapsto \{\text{bisimp. maps }
z^*(\Delta[p]\times\Delta[q])\to X\}\}\\
&= \{[p],[q]\mapsto \int_{([s],[t])}
\mathrm{Ens}(\Delta([t]\sqcup[s],[p])\times\Delta([t],[q]),X_{(s,t)})\} \,.
\end{align*}

Let $\eta_*\colon X\to z_*X$ be the natural transformation associated
with $\eta$.
Notice that $\eta_*$ maps $X_{(0,q)}$ isomorphically to
$(z_*X)_{(0,q)}$ for all $q\geq0$, so $(z_*X)_{(0,q)} \cong X_{(0,q)}$.

\begin{lem} \label{lem:split-mono}
In the homotopy category (with respect to maps that
become weak equivalences upon realization), %jr1804
$\eta_*\colon X\to z_*X$ is a split monomorphism.
\end{lem}
\begin{proof}
  By formal considerations the diagram
%%%br1004 changed adjunction to counit
$$\xymatrix{
  z^*X\ar[r]^{z^*\eta_*}
\ar[dr]_{\eta^*} & z^*z_*X\ar[d]^{\text{counit}}\\
  &X
}
$$
commutes, and $\eta^*$ is a weak equivalence after realization. %jr1804
Hence
$z^*\eta_*$ (and so $\eta_*$) is a split monomorphism in the homotopy category.
\end{proof}

\subsection{The bar construction on matrices} \label{subsec:onesided}

Let $\R$ be a \sbc such that all morphisms are isomorphisms
%% JR3-1: added to make the next lemma work
and each translation functor is faithful.
%,
%and the  only object  $X$ for which there is a map $X \oplus A \to A$
%is $X=0_\R$. %Assume that  $Aut(0_\R)$ is trivial.

Consider the one-sided bar construction $B(*,GL_n(\R),GL_n(T\R))$.
%%%
% In the following, $0$ and $1$ are short for zero resp.~unit matrices
% over $\R$ of varying size. % moved to next subsection, jr1804
Viewing $T\R$ as a simplicial category we get that
$B(*,GL_n(\R),GL_n(T\R))$  is a bisimplicial
category. We are going to show that
$$B(*,GL(\R),GL(T\R))\cong\colim\limits_{n} B(*,GL_n(\R),GL_n(T\R))$$ is
contractible, and  it is enough to show that $B(*,GL(\R),GL(T_\ell\R))$
is contractible for each $\ell$. %jr1804

%%% JR: Do we prefer "matrix" or "matricial" as adjective?
To ease readability, we will abandon the cumbersome $\oplus$ and
$\otimes$ in favor of the more readable $+$ and $\cdot$ --- reminding us
of the matrix nature of our efforts.

Fix $\ell$ once and for all, and let $B^n =B(*,GL_n(\R),GL_n(T_\ell\R))$.
An object in $B^n_q$ is a collection $m_{ij}$ of $n\times n$ matrices
in $\R$  for $0\leq i<j\leq q$, %jr1804
and for each $0\leq i\leq q$ a matrix
$m_{i\infty}$  in $T_\ell\R$, together with suitably compatible structural
isomorphisms $m_{ijk}\colon m_{ij}\cdot m_{jk}\to m_{ik}$.  The
matrices are drawn  from the ``weakly invertible components''.  The matrices
$m_{i\infty}$ and the structural isomorphisms relating these need
special  attention.  Each entry is in $T_\ell\R$, so $m_{i\infty}$ can be
viewed as  a pair $m_{i\infty}^\pm$ of matrices, and the structural
isomorphism  $m_{ij\infty}\colon m_{ij}\cdot m_{j\infty}\to
m_{i\infty}$  is a  tuple
$(m^\pm_{ij\infty},\phi^1_{ij\infty},\dots,
\phi^\ell_{ij\infty})$, where the 
%%%br1004 changed l to r
$\phi^r_{ij\infty}\colon  x^r_{ij\infty}\to x^{r-1}_{ij\infty}\in
M_n(\R)$ for $r = 1, \dots, \ell$ are matrices of isomorphisms, and the
$m^\pm_{ij\infty}\colon
 m_{ij}\cdot  m_{j\infty}^\pm + x^0_{ij\infty}\to m_{i\infty}^\pm$
are isomorphisms.

The assumed %jr1804
commutativity of
$$
\xymatrix{ {(m_{ij} \cdot m_{jk})\cdot m_{k\infty}}
\ar[rr]^{\cong} \ar[d]_{m_{ijk} \cdot \id} &{}& {m_{ij} \cdot
(m_{jk}\cdot m_{k\infty})}\ar[d]^{\id \cdot
m_{jk\infty}}\\
{m_{ik} \cdot m_{k\infty}} \ar[r]_{m_{ik\infty}}&{{m_{i\infty}}}&
{m_{ij}\cdot m_{j\infty}} \ar[l]^{m_{ij\infty}}
}
$$
says that two morphisms from $(m_{ij} \cdot m_{jk})\cdot
m_{k\infty}$ agree: 
%%%br1004 changed l to r
one is an isomorphism with source $(m_{ij} \cdot m_{jk})\cdot
m_{k\infty} + x^r_{ik\infty}$, the other one is an isomorphism
with source $(m_{ij} \cdot m_{jk})\cdot
m_{k\infty} + m_{ij}\cdot  x^r_{jk\infty}+
x^r_{ij\infty}$. Therefore we obtain the following equality.
\begin{lem}\label{1identity}
%Coherence forces
%, among other things, that
In the situation above one has the identity
$$x^r_{ik\infty}=m_{ij}\cdot  x^r_{jk\infty}+ x^r_{ij\infty}$$
%%%%%%0111 for $l=0,1,\dots,m$
for $r=0,1,\dots,\ell$,
%(and likewise for the  $\phi$'s)
and the diagram
$$\xymatrix{
m_{ij} \cdot
(m_{jk}\cdot  m_{k\infty}^\pm)+x^0_{ik\infty}
\ar@{=}[r]&
m_{ij}\cdot (m_{jk}\cdot m_{k\infty}^\pm) +
m_{ij}\cdot x^0_{jk\infty} + x^0_{ij\infty}
\ar[d]^{\id \cdot m_{jk\infty}^\pm+\id}\\
(m_{ij}\cdot m_{jk})\cdot m_{k\infty}^\pm+x^0_{ik\infty}
\ar[u]^{\cong}
\ar[d]_{m_{ijk}\cdot \id + \id}&
m_{ij}\cdot m_{j\infty}^\pm+ x^0_{ij\infty}
\ar[d]^{m_{ij\infty}^\pm}
\\
m_{ik}\cdot m_{k\infty}^\pm+x^0_{ik\infty}\ar[r]^{m_{ik\infty}^\pm}&
m_{i\infty}^\pm
}
$$
commutes.
\end{lem}
Here the map $\id \cdot m_{jk\infty}^\pm$ already incorporates
the distributivity isomorphism, as specified in Lemma
\ref{lem:R-module}.

A morphism $\alpha\colon m\to \tilde{m}$ in $B^n_q$ consists of an
%% JR3-1: Reorganized grouping
$n \times n$ matrix of maps
$\alpha_{ij}\colon  m_{ij}\to \tilde{m}_{ij}$ in $\R$ for
$0\leq i<j\leq q$, and of morphisms
$(\alpha_{i\infty}^\pm,\psi^1_{i\infty},\dots,
\psi^\ell_{i\infty})\colon  m_{i\infty}^\pm\to \tilde{m}_{i\infty}^\pm$
in $T_\ell\R$ for $0 \leq i \leq q$, all
compatible with the  structure maps of $m$ and $\tilde{m}$. %%%%
%%%br1004 reworded slightly to explain the \xi and changed l to r
Thus there are matrices of objects $\xi^r_{i\infty}$ of $\R$ for
$0 \leq r \leq \ell$, each %jr1804
$\alpha_{i\infty}^\pm$ is a map $m_{i\infty}^\pm + \xi^0_{i\infty}
\ra \tilde{m}_{i\infty}^\pm$, and the $\psi^r_{i\infty}$ for
$r = 1, \dots, \ell$ are maps
$\xi^r_{i\infty}\to\xi^{r-1}_{i\infty}$ of matrices in $\R$.
%%%br1004

%In particular, there are objects $\xi_{i\infty}$ such that
%$\alpha_{i\infty}^\pm
%\colon m_{i\infty}^\pm + \xi_{i\infty} \ra \tilde{m}_{i\infty}^\pm$.

The compatibility %jr1804
condition
$\alpha_{i\infty}m_{ij\infty}=\tilde m_{ij\infty}(\alpha_{ij}\cdot\alpha_{j\infty})$
allows us to draw the following conclusion.
%%%br1004 changed l to r
\begin{lem}\label{2identity}
  In the situation above one has the identity
$$x_{ij\infty}^r+\xi^r_{i\infty}=
  m_{ij}\cdot\xi^r_{j\infty}+\tilde x_{ij\infty}^r$$
%%%%%%%%%0111
for each $r = 0, \dots, \ell$.
%%%br1004
\end{lem}

\newcommand{\li}{\text{inc}}
\newcommand{\lj}{\text{jnc}}
\newcommand{\lk}{\text{knc}}
\newcommand{\lnc}{\text{lnc}}

\subsection{Start of the proof that $B(*,GL(\R),GL(T_\ell\R))$ is
contractible} \label{subsec:start}
In the following, $0$ and $1$ are short for zero resp.~unit matrices
over $\R$ of varying size. % jr1804 moved here
We will show that
$$ %jr1804
\colim\limits_{n} B^n=B(*,GL(\R),GL(T_\ell\R))
$$
is
contractible by showing that each matrix stabilization functor
$\text{in} \colon B^n\to B^{2n}$ is trivial in the homotopy category.
Here $\text{in}(m) = \left[\begin{smallmatrix}
    m&0\\0&1\end{smallmatrix} \right]$.

We regard the simplicial categories $B^n$ and $B^{2n}$ as bisimplicial
sets, by way of their respective nerves $NB^n$ and $NB^{2n}$.
To be precise, the $(p,q)$-simplices of $NB^{2n}$ are $N_p B_q^{2n}$.
By Lemma \ref{lem:split-mono} it then suffices to show that the composite
map $\li = \eta_* \circ \text{in} \colon NB^n \to z_* NB^{2n}$
is trivial in the homotopy category.
As remarked above, $z_*(NB^{2n})_{(0,q)} \cong
  (NB^{2n})_{(0,q)} = N_0 B^{2n}_q$, so the subdivision
operator $z_*$ does not make any difference before
we start to consider positive-dimensional simplices ($p>0$)
in the nerve direction.

Seeing that the
image lies in a single path component is easy: if $m\in
N_0 B^n_0 = \ob GL_n(T_\ell\R)$ then there is a path
$$ \left[
  \begin{matrix}
    m&  0\\
    0 & 1
  \end{matrix}\right]
\to \left[
  \begin{matrix}
    m&  m^-\\
    0 & 1
  \end{matrix}\right]\to
\left[
  \begin{matrix}
    m&  m^-\\
    (1,1) & 1
  \end{matrix}\right]\gets\left[
  \begin{matrix}
    1&  0\\
    (0,1) & 1
  \end{matrix}\right] \,.
$$
The first arrow represents the one-simplex in the bar direction
given by the matrix multiplication
% ``by multiplication by''
$$\left[
  \begin{matrix}
    1&  m^-\\
    0 & 1
  \end{matrix}\right] \cdot \left[
  \begin{matrix}
    m&  0\\
    0 & 1
  \end{matrix}\right] =  \left[
  \begin{matrix}
    m&  m^-\\
    0 & 1
  \end{matrix}\right] \in GL_{2n}(T_\ell\R) \,.$$
The second arrow represents the one-simplex in the nerve direction
induced by the morphism %jr1804
$0 = (0,0) \to
(1,1)\in T_0\R \subset T_\ell\R$.
The third map represents the one-simplex in the bar
direction given by multiplication by
%``by multiplication by''
$$\left[
  \begin{matrix}
    m^+&  m^-\\
    1 & 1
  \end{matrix}\right]\in GL_{2n}(\R) \,.$$

The rest of this section extends this path to a full homotopy, from
$\li$ via maps $\lj$ and $\lk$ to a constant map $\lnc$.
%%((write more))

\subsection{The homotopic maps $\li$ and $\lj$} \label{subsec:basiccon}

Recall that $\ell \geq 0$ is fixed,
$B^n = B(*, GL_n(\R), GL_n(T_\ell\R))$ is the simplicial category
given by the one-sided bar construction, and $NB^n \colon
[p],[q] \mapsto N_p B^n_q$ is the bisimplicial set
given by its degreewise nerve.  We already %jr1804
let $\li \colon NB^n \to z_* NB^{2n}$
be the composite of the matrix stabilization map $\text{in} \colon
NB^n \to NB^{2n}$ and the natural map $\eta_* \colon NB^{2n} \to z_* NB^{2n}$.

There is another map $\lj\colon NB^n\to z_*NB^{2n}$ which is homotopic to
$\li$.
On $N_0B_q^n$ it is easy to describe: if $m\in N_0B_q^n$, we declare
that
$X(m)$ is given by
$$X(m)_{ij}=
\begin{cases}
  \left[
  \begin{matrix}
    1&  m^-_{i\infty}\\
    0 & 1
  \end{matrix}\right]& \text{if $i<j=\infty$}\\
  & \\
  \left[
  \begin{matrix}
    1&  x^0_{ij\infty}\\
    0 & 1
  \end{matrix}\right]& \text{if $i<j<\infty$}
\end{cases}
$$
and let
$$\lj(m)=X(m)\cdot\li(m)
%\left[
%  \begin{matrix}
%    m&  m^-\\
%    0 & 1
%  \end{matrix}\right]
\in N_0B_q^{2n} = z_*(NB^{2n})_{(0,q)} \,.$$
Here $$
\lj(m)_{ij}=
\begin{cases}
  \left[
  \begin{matrix}
    m_{i\infty}&  m^-_{i\infty}\\
    0 & 1
  \end{matrix}\right]& \text{if $i<j=\infty$}\\
& \\
\left[
  \begin{matrix}
    m_{ij}&  x^0_{ij\infty}\\
    0 & 1
  \end{matrix}\right]& \text{if $i<j<\infty$}
\end{cases}
$$
with
$\lj(m)_{ijk}\colon \lj(m)_{ij}\cdot\lj(m)_{jk} \rightarrow
\lj(m)_{ik}$
being the isomorphisms induced by $m_{ijk}$ as follows:
for $k<\infty$ we use the
identity $x^0_{ik\infty}=m_{ij}\cdot x^0_{jk\infty}+x^0_{ij\infty}$ from Lemma
\ref{1identity} and obtain
$$\left[
  \begin{matrix}
    m_{ijk}& \id\\
    \id& \id
  \end{matrix}\right] \colon \left[
  \begin{matrix}
    m_{ij}&  x^0_{ij\infty}\\
    0 & 1
  \end{matrix}\right]\cdot\left[
  \begin{matrix}
    m_{jk}&  x^0_{jk\infty}\\
    0 & 1
  \end{matrix}\right]=\left[
  \begin{matrix}
    m_{ij}\cdot m_{jk}&  m_{ij}\cdot x^0_{jk\infty}+x^0_{ij\infty}\\
    0 & 1
  \end{matrix}\right]
%\left[
%  \begin{matrix}
%    m_{ij}\cdot m_{jk}&  x^\ell_{ik\infty}\\
%    0 & 1
%  \end{matrix}\right]
\to\left[
  \begin{matrix}
    m_{ik}&  x^0_{ik\infty}\\
    0 & 1
  \end{matrix}\right]
$$
and for $k=\infty$ we use the string of isomorphisms
$$  \left[\begin{matrix} x^\ell_{ij\infty}& 0 \\ 0& 0 \end{matrix}\right]
  \rightarrow \ldots \rightarrow
  \left[\begin{matrix} x^0_{ij\infty}& 0 \\ 0& 0 \end{matrix}\right]$$
together with the isomorphism
\begin{align*}
\left[
  \begin{matrix}
    m_{ij\infty}& m^-_{ij\infty}\\
    \id& \id
  \end{matrix}\right] &\colon \left[
  \begin{matrix}
    m_{ij}&  x^0_{ij\infty}\\
    0 & 1
  \end{matrix}\right] \cdot \left[
  \begin{matrix}
    m_{j\infty}&  m^-_{j\infty}\\
    0 & 1
  \end{matrix}\right] + \left[
  \begin{matrix}
    x^0_{ij\infty}& 0\\
    0 & 0
  \end{matrix}\right] \\
&= \left[
  \begin{matrix}
    m_{ij}\cdot m_{j\infty} + x^0_{ij\infty}&  m_{ij}\cdot
m^-_{j\infty}+x^0_{ij\infty}\\
    0 & 1
  \end{matrix}\right]\to\left[
  \begin{matrix}
    m_{i\infty}&  m^-_{i\infty}\\
    0 & 1
  \end{matrix}\right] \,.
\end{align*}
We notice that the $T_\ell$-direction does not add any complications 
%%%br1004 changed but into other than
other than notational.  This continues to be true in general, so we simplify
notation by considering only the case $\ell=0$.

The relevant complications arise when one starts moving in the nerve
direction.
%%%
As the construction of the map $\lj$ is quite involved, we
will give some examples first. The impatient reader can skip this part and
restart reading in Subsection \ref{subsec:generalcon} where the formula in
the general case is given.

As an illustration, let $\ell=0$, $p=2$ and $q=0$ so that
$$
\begin{CD}
  m@=(m^0@<{(\xi^1,\alpha^1)}<< m^1@<{(\xi^2,\alpha^2)}<<
  m^2)\in N_2B_0^n = N_2 GL_n(T_0\R) \,.
\end{CD}
$$
Then $\lj(m)$ is captured by the picture
$$\xymatrix{
{\left[
    \begin{smallmatrix}
      m^2&(m^2)^-\\0&1
    \end{smallmatrix}
    \right]}\ar[r]^{{
    \left[\begin{smallmatrix}
1& \xi^2\\0&1
\end{smallmatrix}\right]
}}\ar[dr]_{{
    \left[\begin{smallmatrix}
1& \xi^2 + \xi^1\\0&1
\end{smallmatrix}\right]
}}&{\left[
    \begin{smallmatrix}
      m^2& (m^2)^- + \xi^2\\0&1
    \end{smallmatrix}
    \right]}\ar[d]^{{
    \left[\begin{smallmatrix}
1& \xi^1\\0&1
\end{smallmatrix}\right]
}}\ar[r]&{\left[
    \begin{smallmatrix}
      m^1&(m^1)^-\\0&1
    \end{smallmatrix}
    \right]}\ar[d]^{{
    \left[\begin{smallmatrix}
1& \xi^1\\0&1
\end{smallmatrix}\right]
}}\\
&{\left[
    \begin{smallmatrix}
      m^2& (m^2)^- + \xi^2 + \xi^1 \\
      0&1
    \end{smallmatrix}
    \right]}\ar[r]\ar[dr]&{\left[
    \begin{smallmatrix}
      m^1 & (m^1)^- + \xi^1\\0&1
    \end{smallmatrix}
    \right]}\ar[d]\\
&&{\left[
    \begin{smallmatrix}
      m^0&(m^0)^-\\0&1
    \end{smallmatrix}
    \right]}}$$
where the bar direction is written in the ``$
\begin{CD}
  g@>m>> mg
\end{CD}
$'' form, and the unlabeled arrows correspond to the nerve direction,
with  entries consisting of the appropriate $\alpha$'s.

An even more complicated example, essentially displaying all the
%% JR3-1: Varied the language
complexity of the general case: $\ell=0$, $p=q=1$, and
$\alpha\colon m^1\to m^0\in N_1B^n_1$ with $(m^u)^\pm_{01\infty}\colon
 (m^u)_{01}\cdot  (m^u)_{1\infty}^\pm + (x^u)^0_{01\infty}\to
 (m^u)_{0\infty}^\pm$ (for $u=0,1$ running in the nerve direction),
 $\alpha_{01}\colon m_{01}^1\to m_{01}^0$
and
 $\alpha_{i\infty}^\pm\colon (m^1)_{i\infty}^\pm+\xi_{i\infty}\to
 (m^0)_{i\infty}^\pm$.

Then $\lj(m)$ is the map from
$$z^*(\Delta[1]\times\Delta[1]) =\quad
\xymatrix{((0,0),0)&((0,1),0)\ar[l]&((1,1),0)\ar[l]\\
((0,0),1)\ar[u]&((0,1),1)\ar[l]\ar[u]&((1,1),1)\ar[l]\ar[u]\ar[ul]}$$
sending
%% JR3-1: nondegenerate seems unnecessary
the %(nondegenerate)
$((0,1),0)\gets((1,1),0)\gets((1,1),1)$ simplex to
$$
\begin{matrix}
  \left[
    \begin{smallmatrix}
      1& \xi_{0\infty} \\
      0&1
    \end{smallmatrix}
    \right]
&
\left[
    \begin{smallmatrix}
      m^1_{01}& x^1_{01\infty}+\xi_{0\infty}\\
      0&1
    \end{smallmatrix}
    \right]
&
\left[
    \begin{smallmatrix}
      m^1_{0\infty}& (m^1)^-_{0\infty}+\xi_{0\infty}\\
      0&1
    \end{smallmatrix}
    \right]\\
&

\left[
    \begin{smallmatrix}
      m^1_{01}& x^1_{01\infty}\\
      0&1
    \end{smallmatrix}
    \right]
&
\left[
    \begin{smallmatrix}
      m^1_{0\infty}& (m^1)^-_{0\infty}\\
      0&1
    \end{smallmatrix}
    \right]\\
&
&
\left[
    \begin{smallmatrix}
      m^1_{1\infty}& (m^1)^-_{1\infty}\\
      0&1
    \end{smallmatrix}
    \right]\\
\end{matrix}\in N_0 B_2^{2n} \,,
$$
the $((0,1),0)\gets((0,1),1)\gets((1,1),1)$ simplex to
$$
\begin{matrix}
\left[
    \begin{smallmatrix}
      m^1_{01}& x^0_{01\infty}\\
      0&1
    \end{smallmatrix}
    \right]
&
\left[
    \begin{smallmatrix}
      m^1_{01}& x^1_{01\infty}+\xi_{0\infty}\\
      0&1
    \end{smallmatrix}
    \right]
&
\left[
    \begin{smallmatrix}
      m^1_{0\infty}& (m^1)^-_{0\infty}+\xi_{0\infty}\\
      0&1
    \end{smallmatrix}
    \right]\\
&
  \left[
    \begin{smallmatrix}
      1& \xi_{1\infty} \\
      0&1
    \end{smallmatrix}
    \right]
&
\left[
    \begin{smallmatrix}
      m^1_{1\infty}& (m^1)^-_{1\infty}+\xi_{1\infty}\\
      0&1
    \end{smallmatrix}
    \right]\\
&
&
\left[
    \begin{smallmatrix}
      m^1_{1\infty}& (m^1)^-_{1\infty}\\
      0&1
    \end{smallmatrix}
    \right]\\
\end{matrix}\in N_0 B_2^{2n} \,,
$$
(here the identity from Lemma \ref{2identity} is used)
and the $(1,1)$-simplex
$\xymatrix{((0,0),0)&((0,1),0)\ar[l]\\
((0,0),1)\ar[u]&((0,1),1)\ar[l]\ar[u]}$
to

%%%140809 CD just screws up my editor...
% $$
% \begin{CD}
%   \begin{matrix}
% \left[
%     \begin{smallmatrix}
%       m^0_{01}& x^0_{01\infty}\\
%       0&1
%     \end{smallmatrix}
%     \right]
% &
% \left[
%     \begin{smallmatrix}
%       m^0_{0\infty}& (m^0)^-_{0\infty}\\
%       0&1
%     \end{smallmatrix}
%     \right]\\
% &
% \left[
%     \begin{smallmatrix}
%       m^0_{1\infty}& (m^0)^-_{1\infty}\\
%       0&1
%     \end{smallmatrix}
%     \right]\\
% \end{matrix}
% @<{\begin{matrix}
% \left[
%     \begin{smallmatrix}
%       \alpha_{01}& \id\\
%       \id & \id
%     \end{smallmatrix}
%     \right]
% &
% \left[
%     \begin{smallmatrix}
%       \alpha_{0\infty}& \alpha^-_{0\infty}\\
%       \id & \id
%     \end{smallmatrix}
%     \right]\\
% &
% \left[
%     \begin{smallmatrix}
%       \alpha_{1\infty}& \alpha^-_{1\infty}\\
%       \id & \id
%     \end{smallmatrix}
%     \right]\\
% \end{matrix}
% }<<
%   \begin{matrix}
% \left[
%     \begin{smallmatrix}
%       m^1_{01}& x^0_{01\infty}\\
%       0&1
%     \end{smallmatrix}
%     \right]
% &
% \left[
%     \begin{smallmatrix}
%       m^1_{0\infty}& (m^1)^-_{0\infty}+\xi_{0\infty}\\
%       0&1
%     \end{smallmatrix}
%     \right]\\
% &
% \left[
%     \begin{smallmatrix}
%       m^1_{1\infty}& (m^1)^-_{1\infty}+\xi_{1\infty}\\
%       0&1
%     \end{smallmatrix}
%     \right]\\
% \end{matrix}
% \end{CD}
%  \in N_1B_1^{2n} \,.
% $$

$$
\xymatrix{
{\begin{array}{cc}
\begin{matrix}
\left[
    \begin{smallmatrix}
      m^0_{01}& x^0_{01\infty}\\
      0&1
    \end{smallmatrix}
    \right]
&
\left[
    \begin{smallmatrix}
      m^0_{0\infty}& (m^0)^-_{0\infty}\\
      0&1
    \end{smallmatrix}
    \right]\\
&
\left[
    \begin{smallmatrix}
      m^0_{1\infty}& (m^0)^-_{1\infty}\\
      0&1
    \end{smallmatrix}
    \right]\\
\end{matrix}
\end{array}}
& & & & {\begin{array}{cc}
  \begin{matrix}
\left[
    \begin{smallmatrix}
      m^1_{01}& x^0_{01\infty}\\
      0&1
    \end{smallmatrix}
    \right]
&
\left[
    \begin{smallmatrix}
      m^1_{0\infty}& (m^1)^-_{0\infty}+\xi_{0\infty}\\
      0&1
    \end{smallmatrix}
    \right]\\
&
\left[
    \begin{smallmatrix}
      m^1_{1\infty}& (m^1)^-_{1\infty}+\xi_{1\infty}\\
      0&1
    \end{smallmatrix}
    \right]\\
\end{matrix}
\end{array}}
 \ar[llll]_{\begin{array}{cc}
\begin{matrix}
\left[
    \begin{smallmatrix}
      \alpha_{01}& \id\\
      \id & \id
    \end{smallmatrix}
    \right]
&
\left[
    \begin{smallmatrix}
      \alpha_{0\infty}& \alpha^-_{0\infty}\\
      \id & \id
    \end{smallmatrix}
    \right]\\
&
\left[
    \begin{smallmatrix}
      \alpha_{1\infty}& \alpha^-_{1\infty}\\
      \id & \id
    \end{smallmatrix}
    \right]\\
\end{matrix} \end{array}}
} \in N_1B_1^{2n} \,.
$$

Here we have employed the formula
$x_{01\infty}^1+\xi_{0\infty}=m_{01}^1\cdot\xi_{1\infty}+x_{01\infty}^0$
of Lemma \ref{2identity}.

% ((Continue tomorrow: on sheet 070510: these explicit calculations in
% low degrees should probably not make it to the final version, but
% should stay there  as
% an aid till the coauthors have read through the argument))

\subsection{General version of $\lj$} \label{subsec:generalcon}
Consider %In general, $\lj$ sends
\begin{equation}\label{eq:m}
% \begin{CD}
%  m@= (m^0
% @<{\alpha^1}<<m^1@<<<\dots@<{\alpha^p}<<m^p)
% \in N_pB^n_q \,.
% \end{CD}
\xymatrix@1{
{m=(m^0} & \ar[l]_(0.3){\alpha^1} {m^1} & \ar[l]_{\alpha^2} {\ldots} &
\ar[l]_{\alpha^p} {m^p)}} \in  N_pB^n_q \,.
\end{equation}

%%%br1004 changed l into r
Then $(\alpha^u)_{i\infty}$ is given by the tuple
$$((\alpha^u)_{i\infty}^\pm\colon
(m^u)_{i\infty}^\pm + (\xi^u)^0_{i\infty} \to (m^{u-1})_{i\infty}^\pm \,,
\{ (\psi^u)^r_{i\infty} \colon (\xi^u)^r_{i\infty} \to
(\xi^u)^{r-1}_{i\infty} \})$$
for $u=1, \dots, p$, $r = 1, \dots, \ell$, 
%%%br1004
but we simplify notation by setting
$\xi^u_{i\infty}=(\xi^u)_{i\infty}^0$,
$x^u_{ij\infty}=(x^u)_{ij\infty}^0$, and ignoring the $\psi$'s. Then
$\lj$ sends an $m$ as in \eqref{eq:m}  to the simplex
$\lj(m)\in z_*(NB^n)_{pq}$ with value at the $((a,b),c)$-vertex in
$z^*(\Delta[p]\times\Delta[q])$
%$\Delta([0],[q])\times \Delta([0]\sqcup[0],[p])$
given by
$$\left[
  \begin{matrix}
    (m^b)_{c\infty}&  (m^b)^-_{c\infty} +
\xi^b_{c\infty} + \dots + \xi^{a+1}_{c\infty}\\
    0 & 1
  \end{matrix}\right]
\in GL_{2n}(T_\ell\R)
$$
%%%
with the convention that the $\xi$'s only occur if $a+1\leq b$.
Higher simplices are given by the structural isomorphisms in $m$.
Note that  the elements in the off-diagonal blocks are actually all in $\R$.

More precisely, a triple $(\phi,b,\psi)$ where  $\phi\colon [r]\to [p]$
and $\psi\colon [r]\to[q]$ are in $\Delta$ and $\phi(r)\leq b\leq p$,
determines a $(0,r)$-simplex in $z^*(\Delta[p]\times\Delta[q])$, %%%
because $z^*(\Delta[p]\times\Delta[q])_{(0,r)} =
\Delta([r+1],[p])\times\Delta([r],[q])$ and $\phi$ together with $b$
determine an element in the first factor. We see
that  $\lj(m)(\phi,b,\psi)\in N_0B^{2n}_r$ is the element whose $(0\leq
i<j\leq r)$- and $(0\leq i\leq r<j=\infty)$-entries are
$$\left[
    \begin{matrix}
      (m^b)_{\psi(i)\psi(j)}&
  x^{\phi(j)}_{\psi(i)\psi(j)\infty} + \xi^{\phi(j)}_{\psi(i)\infty}
  + \dots + \xi^{\phi(i)+1}_{\psi(i)\infty} \\
      0&1
    \end{matrix}
    \right]$$
and
$$\left[
    \begin{matrix}
      (m^b)_{\psi(i)\infty}&
  (m^b)^-_{\psi(i)\infty} + \xi_{\psi(i)\infty}^b
  + \dots + \xi_{\psi(i)\infty}^{\phi(i)+1}\\
      0&1
    \end{matrix}
    \right] \,,
$$
respectively.  (As before, the $\xi$'s only occur
when $\phi(i)+1 \leq \phi(j)$ and $\phi(i)+1 \leq b$, respectively.)

Moving in the (nerve =) $b$-direction is easy %%%
because it amounts to connecting
two values $\lj(m)(\phi,b,\psi)$ and $\lj(m)(\phi,b',\psi)$ by morphisms.
Since this is determined by %its
the one-skeleton, it is enough to describe the case $b'=b-1<b$.  On the $(0\leq
i<j\leq r)$-entries it is induced by $(\alpha^b)_{\psi(i)\psi(j)}\colon
(m^b)_{\psi(i)\psi(j)}\to (m^{b-1})_{\psi(i)\psi(j)}$ (in the upper
left hand corner, and otherwise the
identity), and on the $(0\leq i\leq r<j=\infty)$-entries it is given by
$(\xi^b_{\psi(i)\infty}, (\alpha^b)_{\psi(i)\infty})
\colon (m^b)_{\psi(i)\infty} \to
(m^{b-1})_{\psi(i)\infty}$ and $(\alpha^b)^-_{\psi(i)\infty}\colon
(m^b)^-_{\psi(i)\infty}+\xi^b_{\psi(i)\infty}\to
(m^{b-1})^-_{\psi(i)\infty}$ (in the upper row, and otherwise the identity).

Checking that this is well defined and
simplicial amounts to the same kind of checking as we have already
encountered, using the same identities.
One should notice that at no time during the verifications is the
symmetry of addition used.
%%%
It is used, however, for the isomorphism
that renders matrix
multiplication associative up to isomorphism.

The simplicial homotopy from $\li$ to $\lj$ is gotten by %the obvious
multiplications (in the  bar direction) by matrices of the form
$$\left[
    \begin{matrix}
      1&
x^{\phi(j)}_{\psi(i)\psi(j)\infty}+\xi^{\phi(j)}_{\psi(i)\infty}+\dots+
\xi^{\phi(i)+1}_{\psi(i)\infty}\\
      0&1
    \end{matrix}
    \right]
\qquad\text{and}\qquad
\left[
    \begin{matrix}
      1& (m^b)^-_{\psi(i)\infty} + \xi_{\psi(i)\infty}^b
  + \dots + \xi^{\phi(i)+1}_{\psi(i)\infty}\\
      0&1
    \end{matrix}
    \right] \,.
$$

\subsection{The homotopic maps $\lk$ and $\lnc$}

Consider the following variant $\lk$ of the map $\lj$:
using the same notation as for $\lj$, when evaluated on
$(\phi,b,\psi)$ where   $\phi\colon [r]\to [p]$ and $\psi\colon [r]\to[q]$
are in $\Delta$ and $\phi(r)\leq b\leq p$, $\lk(m)(\phi,b,\psi)\in
N_0B^{2n}_r$  is the element whose $(0\leq i<j\leq r)$- and  $(0\leq i\leq
r<j=\infty)$-entries are
$$\left[
    \begin{matrix}
      (m^b)_{\psi(i)\psi(j)}&
x_{\psi(i)\psi(j)\infty}^{\phi(j)}+\xi_{\psi(i)\infty}^{\phi(j)}+\dots+
\xi_{\psi(i)\infty}^{\phi(i)+1}\\
      0&1
    \end{matrix}
    \right]$$
and
$$
\left[
  \begin{matrix}
    (m^b)_{\psi(i)\infty} +
      (\xi_{\psi(i)\infty}^b + \dots + \xi_{\psi(i)\infty}^{\phi(i)+1},
        \xi_{\psi(i)\infty}^b+\dots+\xi_{\psi(i)\infty}^{\phi(i)+1})&
    (m^b)^-_{\psi(i)\infty} +
      \xi_{\psi(i)\infty}^b + \dots + \xi_{\psi(i)\infty}^{\phi(i)+1}\\
      (1,1)&1
  \end{matrix}
\right] \,,
$$
respectively.
%%%br1004 just commented that out
%%%The $(A,B)$-notation is the ``plus-minus'' notation
%%%for objects in $T_\ell\R$.
The entries for $j=\infty$ can be written %jr1804
concisely as
$
\left[\begin{matrix}
  (m^b)_{\psi(i)\infty} + (\Xi,\Xi) & (m^b)^-_{\psi(i)\infty} + \Xi \\
  (1,1) & 1
\end{matrix}\right]
$
where $\Xi = \xi_{\psi(i)\infty}^b + \dots + \xi_{\psi(i)\infty}^{\phi(i)+1}$.
The entries for finite $j$ are the same as for $\lj$.

There is a natural map (in the nerve direction) from $\lj$ to $\lk$
(of the form $(X, \id) \colon (A,B)\to (A+X,B+X)\in T_\ell\R$ --- induced
by the identity), giving a homotopy.

Finally, let $\lnc \colon NB^n\to z_* NB^{2n}$ be
induced by  the constant map sending any matrix to $\left[
\begin{smallmatrix}
  1&0\\(0,1)&1
\end{smallmatrix}\right]
$.
Matrix multiplication yields
\begin{multline*} \left[\begin{matrix}
    (m^b)^+_{\psi(i)\infty} + \xi_{\psi(i)\infty}^b+\dots+
\xi_{\psi(i)\infty}^{\phi(i)+1}
  & (m^b)^-_{\psi(i)\infty} + \xi_{\psi(i)\infty}^b+\dots+
\xi_{\psi(i)\infty}^{\phi(i)+1}\\
    1 & 1
  \end{matrix}\right] \cdot \left[\begin{matrix}
  1&0\\(0,1)&1
\end{matrix}\right] \\
=
\left[
    \begin{matrix}
      (m^b)_{\psi(i)\infty}+(\xi_{\psi(i)\infty}^b+\dots+
\xi_{\psi(i)\infty}^{\phi(i)+1},\xi_{\psi(i)\infty}^b+\dots+
\xi_{\psi(i)\infty}^{\phi(i)+1})& (m^b)^-_{\psi(i)\infty}+
\xi_{\psi(i)\infty}^{c}+\dots+\xi_{\psi(i)\infty}^{\phi(i)+1}\\
      (1,1)&1
    \end{matrix}
    \right] \,.
\end{multline*}
With the same abbreviation as above, this reads
$$\left[\begin{matrix}
(m^b)^+_{\psi(i)\infty} + \Xi & (m^b)^-_{\psi(i)\infty} + \Xi \\ 1 & 1
\end{matrix}\right] \cdot
\left[\begin{matrix}
1 & 0 \\ (0,1) & 1
\end{matrix}\right] =
\left[\begin{matrix}
(m^b)_{\psi(i)\infty} + (\Xi,\Xi) & (m^b)^-_{\psi(i)\infty} + \Xi \\ (1,1) & 1
\end{matrix}\right] \,.$$
We obtain a homotopy from $\lnc$ to $\lk$.

Hence $\li$ is connected by a chain
of homotopies to a constant map.  Since $\eta_* \colon \id
\to z_*$ is a monomorphism in the homotopy category, this  means that
the stabilization map $\text{in} \colon B^n\to B^{2n}$ is homotopically
trivial, and so  $B(*,GL(\R),GL(T_\ell\R))= \colim\limits_{n} B^n$
is contractible for each $\ell\geq0$.  Hence $B(*, GL(\R), GL(T\R))$
is also contractible.
This concludes the proof of Theorem~\ref{thm:main}.

\end{document}